\documentclass[12pt]{amsart}
\usepackage{geometry} 
\usepackage{amstext}
\usepackage{amsmath, amsthm}
\usepackage{amssymb}
\usepackage{graphicx}
\numberwithin{equation}{section}
\newtheorem{thm}{Theorem}[section]
\newtheorem{cor}[thm]{Corollary}
\newtheorem{lem}[thm]{Lemma}
\newtheorem{prop}[thm]{Proposition}

\usepackage{color}
\begin{document}

\title{Cutoff for random walks on graphs with  bottlenecks.}
\author{Ioannis Papageorgiou}
\thanks{\textit{Address:} Neuromat, Instituto de Matematica e Estatistica,
 Universidade de Sao Paulo, 
 rua do Matao 1010,
 Cidade Universitaria, 
 Sao Paulo - SP-  Brasil - CEP 05508-090. 
   \\
 This article was produced as part of the activities of FAPESP  Research, Innovation and Dissemination Center for Neuromathematics (grant 2013/ 07699-0 , S. Paulo Research Foundation); This article  is supported by FAPESP grant   (2017/15587-8)
\\ \text{\  \   \      } \textit{Email:} papyannis@yahoo.com ,  ipapageo@ime.usp.br}
\keywords{cutoff, mixing time.}
\subjclass[2000]{60J10}
\begin{abstract}
We examine the mixing time  for random walks on graphs. In particular we are interested on investigating graphs with bottlenecks. Furthermore,  the cutoff  phenomenon is examined.
\end{abstract}
\maketitle

\section{Introduction.}

Assume $X$ is an irreducible  aperiodic Markov chain on some finite state space. Consider $P$ to be the transition matrix and $\pi$ the stationary distribution. At first define the total variation distance between two measures $\mu$ and $\nu$ to be
$$\vert \vert \mu-\nu\vert \vert = \sup _A\vert\mu(A)-\nu(A)\vert.$$
Then for every $\epsilon>0$ we define the $\epsilon$-total variation mixing time as
$$t_{mix}(\epsilon) = \min \{t\geq 0 : \max_x \vert \vert P^t(x,\dot ) - \pi \vert \vert \leq \epsilon \}.$$
The purpose of the current paper is to calculate upper and lower bounds for the mixing time of an irreducible Markov chain on graphs with bottlenecks. We aim in extending the example presented by Peres and Sousi in \cite{P-S}  to a general result that will include cases not covered yet. In particular, we are interested in determining conditions under which the upper and lower bounds of the mixing time are asymptotically equal.

\begin{figure}
\begin{center}
\includegraphics[width=12.25cm,height=5.25cm]{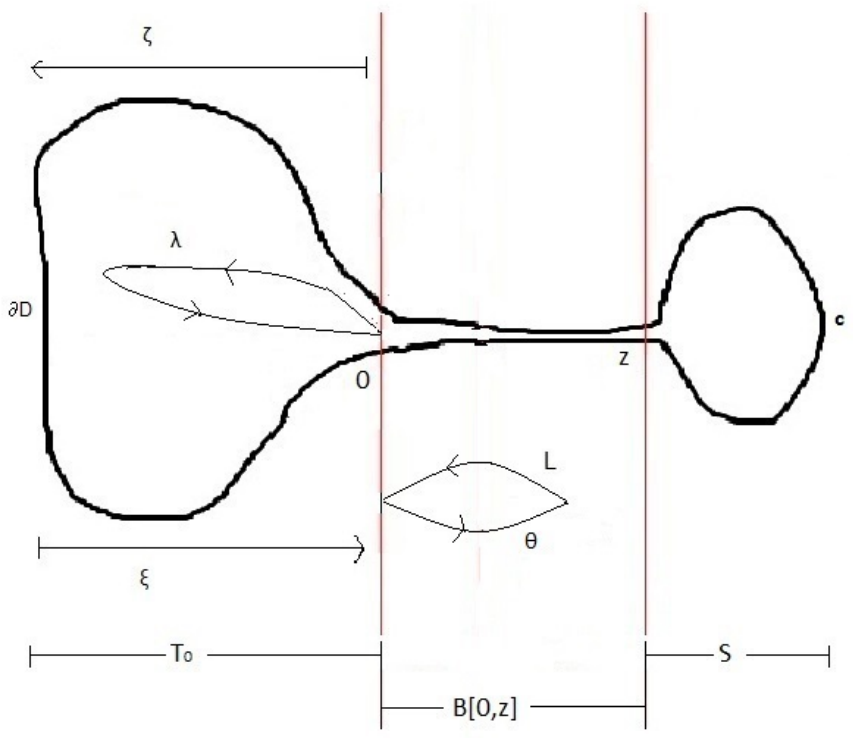}
\caption[Case A]{Case A.}
\label{fig1}
\end{center}
\end{figure}
 Furthermore, some outcomes in relation to the cutoff phenomena are obtained. 

We say that the sequence of chains $X^n $ exhibits total  variation cutoff if for every $0<\epsilon<1$
\begin{align}\label{eqcutoff}\lim_{n \rightarrow \infty}\frac{t_{mix}^n (\epsilon)}{t_{mix}^n (1-\epsilon)}=1.\end{align}
In \cite{P-S}  the first example of a tree was  constructed which exhibits total variation cutoff. The construction of the tree was based on placing a binary tree consisting of $N=n_k^3$ vertices at the origin of a line of $n_k$ points, where $n_k=2^{2^k}$. Then for every $j\in \{[k/ 2],...,k\}$ a  binary tree $T_j$ consisting of $N/n_j$ vertices was placed at distance $n_j$ from the origin.
The purpose of the current paper is to generalise this result. We will consider two cases, one general case referred to as Case A (see figure 1) where we will consider two graphs connected with a bottleneck and another one  referred to as Case B (see figure 2), where  we will   substitute the trees $T_j$ in \cite{P-S} by finite graphs $T_j$ in such a way that a bottleneck is observed between the $T_j$ 's. 
\begin{figure}
\begin{center}
\includegraphics[width=12.25cm,height=5.25cm]{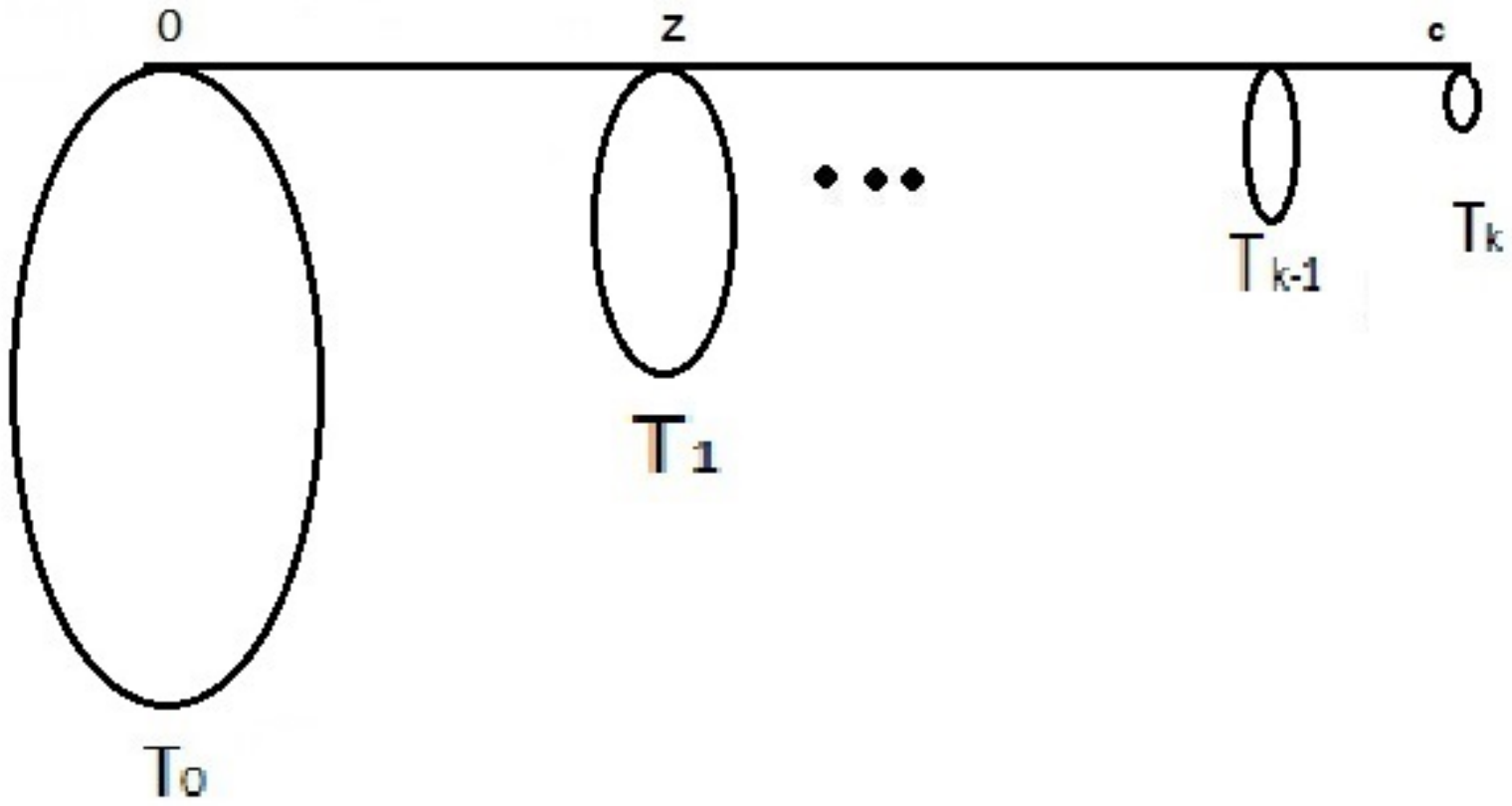}
\caption[Case B]{Case B.}
\label{fig2}
\end{center}
\end{figure}
We will denote by $B[0,z]$  the bottleneck of the graph, the part of the graph between $0$ and $z$, that is the part of the graph between $T_0$ and $S$. All components of the graph, $ T_i, S_i, B[0,z], S$   increase in size as $k$ increases. We will denote  $D=T_0 \cup B[0,z]$ while  the complement of $D$ will be denoted as $S=D^c$. Furthermore, we will denote   $\partial D$  the points of $T_0$ that are further from $0$, i.e. $\partial D:=\left\{  x\in T_0 \text{\ : \ }d(x, 0)=\max_{x\in T_0}d(x, 0)\right\}$, where $d(x,y)$ the graph distance between to points $x,y$ of the graph. In  Case A we will denote $c$ the right end point  of $S$, while in case B, $c$ will be the connecting point between $T_k$ and the rest of the graph. We will consider the size of the bottleneck and $S$ to be relatively small compared with how the size of $T_0$ increases, in such a way that $$(C1) \ \ \ \  \  \  \pi (B[0,z]\cup S)\rightarrow 0 \text{ \ as \ } k\rightarrow \infty. $$

  We define $q$ to be  the  probability at  every return to $T_0$ the random walk to exit  $T_0$ without hitting $\partial D$.    Then,  $p=1-q$, is  the  probability at  every return to $T_0$, the random walk to hit  $\partial D$ of $T_0$ before   exiting $T_0$.   For instance, if $T_0$ is a binary tree, as is the case examined  in \cite{P-S}, $q=1-p\geq \frac{1}{3}$.
 
    By  $\tau_x$ we define the first time we reach $x$ 
$$\tau_x=inf\{t:x_t=x\} .$$ 
Thus, $\tau_z$ denotes the first time we reach the  boundary between $S$ and $D$.  Furthermore, we define  
$$\tau^*_x=inf\{s\geq \tau_0:x_s=x\}. $$
 Through  out this paper we will   use the symbol $\prec$ to denote  domination, i.e. we will write $$A\prec B \Leftrightarrow\lim_{k\rightarrow +\infty} \frac{ A}{B } =0 $$ Furthermore, we will write
$$A\asymp B\Leftrightarrow \ \ni\ \text{constant}\ c \ \text{s.t.} \ A=cB$$
and 
$$A\lesssim\ B\Leftrightarrow \ \ni\ \text{constant}\ c \ \text{s.t.} \ A\leq cB.$$

We will now look at  the main conditions   and results of the paper. We focus on the two distinct  conditions (H1) and (H2) that result in the two theorems of the paper. The remaining set of conditions labelled as (H) and (C) is common in the two theorems.   The main condition states that for some   $\gamma$ and $\delta$ increasing on $k$  the following   holds  
 \begin{align*}(H1) \  \  \   \   \  \ \mathbb{E}_{ c}(\tau_0)+\frac{\gamma}{2}\sqrt{Var_{ c}(\tau_0)}\geq\mathbb{E}_{ \partial D}(\tau_0)+\delta\ \sqrt{Var_{ \partial D}(\tau_0)} \end{align*}
for large $k$, with the growth of $\gamma$ being   constrained  by $$ (H)   (a)\  \    \gamma\sqrt{Var_{ c}(\tau_0)} \prec\mathbb{E}_{ c}(\tau_0).$$  Then we show that $t_{mix}(\epsilon)\sim \mathbb{E}_{ c}(\tau_0)$. This is the results presented in   Theorem \ref{theorem1}. In this category belongs the example presented in \cite{P-S}, as explained in section \ref{secpar1}. Then we look on graphs that (H1) does not hold. It appears that with some additional conditions  the statement of the theorem derived under (H1) still holds true. Assume that the opposite inequality of (H1) holds, that is:   
 \begin{align*}  \mathbb{E}_{ c}(\tau_0)+\frac{\gamma}{2}\sqrt{Var_{ c}(\tau_0)}\leq\mathbb{E}_{ \partial D}(\tau_0)+\delta\ \sqrt{Var_{ \partial D}(\tau_0)} \end{align*}
for large $k$.  We will determine conditions so that an asymptotic estimate of the mixing time can be obtained in this case as well. Denote $t_S=\mathbb{E}_{ c}(\tau_0)-\gamma\sqrt{Var_{ c}(\tau_0)}$, and consider some $\epsilon \in (0,1)$. Assume that for $k$ large,  $\pi(S)+\sqrt{\Phi}<1-\epsilon$ for  $\Phi$ defined as 
\begin{align}\label{defPhi}\Phi=\sum_{x \in S } \frac{\pi(x)}{\pi(S)}P_{x} (t_{S}>\tau_0), \end{align} 
a function that decreases to zero as $k\rightarrow \infty$.
 If  in addition  $D$ and $S$ are such that the following  inequality  (H2) holds for large $k$
\begin{align*}  \frac{\mathbb{E}_{ c}(\tau_0)+\frac{\gamma}{2}\sqrt{Var_{ c}(\tau_0)}}{\sqrt{\Phi}} \geq\mathbb{E}_{ \partial D}(\tau_0)+\delta\ \sqrt{Var_{ \partial D}(\tau_0)} \geq \mathbb{E}_{ c}(\tau_0)+\frac{\gamma}{2}\sqrt{Var_{ c}(\tau_0)},\end{align*}
 then $t_{mix}(\epsilon) \sim \frac{\mathbb{E}_{ c}(\tau_0) }{\sqrt{\Phi}}$. This will be the subject of Theorem \ref{theorem0}.  One should notice that the right inequality of (H2) is nothing else than the inverse of (H1).  Furthermore,  $\sqrt{\Phi}$ is not a constant but  a decreasing function such that $\sqrt{\Phi}\downarrow 0$ as $k\rightarrow \infty$. Specific examples are presented in   sections  \ref{secpar2} and \ref{secpar3}. In this way we can even consider examples that go far from (H1), such that 
$$\frac{\mathbb{E}_{ c}(\tau_0)+\frac{\gamma}{2}\sqrt{Var_{ c}(\tau_0)}}{\mathbb{E}_{ \partial D}(\tau_0)+\delta\ \sqrt{Var_{ \partial D}(\tau_0)}} \rightarrow0 \  \  \ \text{as} \ \ k\rightarrow \infty .$$
  For both cases investigated in theorems \ref{theorem1}
and   \ref{theorem0} cutoff properties are proven in corollaries \ref{cor1} and \ref{cor2}.   

Some further assumptions on the graph. 
 In relation to the bottleneck  we also define 
\begin{align}\label{defofL}L_z=\sum_{s=\tau_0}^{\tau^*_{z}}\mathcal{I}(x_{s-1}\notin T_0,x_s \in T_0)\end{align}  the number  of returns  to $T_{0}$ in the time interval $[\tau_0,\tau^*_{z}]$. We also assume that the part of the graph between $0$ and $z$ is such that there exists an increasing function $h$, $\lim_{x\rightarrow \infty}h(x)= \infty$, such that \begin{align*} 
(C2) \  \  \   \  \  \  \  \  \ P(L_z\leq z)\leq\frac{1}{  h(z)}.
\end{align*} 
For the analogue  of the example in \cite{P-S} where the bottle neck is a line, $h(x)=x$.

We define  $(\theta_{i})$ to be iid variables distributed as the length of a random walk on $B[0,z]$ starting from $0$ conditioning not to hit $z$ and $(\lambda_{i,j})$ to be iid random variables distributed as the length of a random walk from $0$ on $T_0$ conditioning not to hit $\partial T_0$, the edge of $T_0$. For $n_i$ the point where $T_i$ connects with the rest of the graph, we denote for  every $i\geq 1$,   $\partial T_i:=\left\{  x\in T_i \text{\ : \ }d(x, n_i)=\max_{x\in T_i}d(x, n_i)\right\}$, that is the points in $T_i$ that lie at the biggest distance far from $n_i$, and so the rest of the graph  $T_i^c$.  We  will also  assume  \begin{align*} (C3)  \   \   \   \   \  \   \mathbb{E}_0(\tau_{ \partial D})&\leq \mathbb{E}_{ \partial D}(\tau_0) \\    \mathbb{E}_{ \partial T_i}(n_i)&\leq \mathbb{E}_{ \partial D}(\tau_0) \ \  \forall i \end{align*}
for large $k$.  Conditions  (C3) are reasonable enough since they  roughly state that the time needed to traverse  the big  set of vertices with the complex structure $T_0$ is bigger than  the smaller $T_i$'s, as well as, that when in $T_0$, the random walk moves with bigger probability towards the left end side $\partial D$ than towards the bottleneck. The main theorem about the  mixing time follows.
   
 \begin{thm} \label{theorem1} Assume conditions (C) and     that for large $k$ there exist $\gamma$ and $\delta$ such that 
\begin{align}\nonumber (H)   (a)\  \    \gamma\sqrt{Var_{ c}(\tau_0)} \prec\mathbb{E}_{ c}(\tau_0)\text{    \  and \    }(b)  \  \mathbb{E}_{ \partial D}(\tau_0)+\mathbb{E}[ L]  \mathbb{E}[\lambda_{i,j}+\theta_{i}]\prec\gamma\sqrt{Var_{ c}(\tau_0)}\end{align}
and  \begin{align*}(H1) \  \  \   \   \  \ \mathbb{E}_{ c}(\tau_0)+\frac{\gamma}{2}\sqrt{Var_{ c}(\tau_0)}\geq\mathbb{E}_{ \partial D}(\tau_0)+\delta\ \sqrt{Var_{ \partial D}(\tau_0)} \end{align*}
hold.
Then for   every $0<\epsilon < 1$ the mixing time for the random walk on the graph is $$t^k_{mix}(\epsilon)\sim \mathbb{E}_{ c}(\tau_0) $$ \end{thm}

 In relation then to the  cutoff phenomenon, we obtain the following corollary. 
 \begin{cor}\label{cor1}Under the conditions of Theorem \ref{theorem1} the random walk on the graph exhibits variation cutoff, that is for every $0<\epsilon<1$
\begin{align*}\lim_{k \rightarrow \infty}\frac{t_{mix}^k (\epsilon)}{t_{mix}^k (1-\epsilon)}=1.\end{align*}
\end{cor}  

For an example   that    satisfies the conditions of Theorem \ref{theorem1} one can look on graphs similar to the one presented in \cite{P-S} as described in section \ref{secpar1}.

In the case where  condition (H1) is not satisfied then an analogue result still holds as presented on the next theorem. First,  recall that because of (C1), $\lim_{k \rightarrow \infty }\pi(S)=0$.
\begin{thm} \label{theorem0}   Assume  that $\Phi= \sum_{x \in S } \frac{\pi(x)}{\pi(S)}P_{x} (t_{S}>\tau_0)\downarrow 0$ as $k \rightarrow \infty$.  For $\epsilon \in (0,1)$, assume  $k$ large enough so that   $\pi(S)+\sqrt{\Phi}<1-\epsilon$.  Assume  conditions (C) and     that for large $k$ there exist $\gamma$ and $\delta$ such that 
\begin{align}\nonumber(H) \  (a)\ \       \gamma\sqrt{Var_{ c}(\tau_0)} \prec\mathbb{E}_{ c}(\tau_0)\text{   \    and \    }(b)  \  \mathbb{E}_{ \partial D}(\tau_0)+\mathbb{E}[ L]  \mathbb{E}[\lambda_{i,j}+\theta_i])\prec\frac{\gamma\sqrt{Var_{ c}(\tau_0)}}{\sqrt{\Phi}}\end{align} and
\begin{align*}(H2) \  \  \   \   \  \  \frac{\mathbb{E}_{ c}(\tau_0)+\frac{\gamma}{2}\sqrt{Var_{ c}(\tau_0)}}{\sqrt{\Phi}} \geq \mathbb{E}_{ \partial D}(\tau_0)+\delta\ \sqrt{Var_{ \partial D}(\tau_0)} \geq   \mathbb{E}_{ c}(\tau_0)+\frac{\gamma}{2}\sqrt{Var_{ c}(\tau_0)}.\end{align*} 
Then, 
\begin{align*} t^k_{mix}(\epsilon) \sim \frac{\mathbb{E}_{ c}(\tau_0) }{\sqrt{\Phi}}.\end{align*}
 \end{thm}
In sections \ref{secpar2} and \ref{secpar3} two examples  that satisfy the conditions of Theorem \ref{theorem0} are presented.  As an outcome of the theorem, we obtain the following corollary about variation cutoff.  
\begin{cor}\label{cor2}Under the conditions of Theorem \ref{theorem0} the random walk $X^n$ exhibits variation cutoff, that is, for every $0<\epsilon<1$ 
\begin{align*}\lim_{k \rightarrow \infty}\frac{t_{mix}^k (\epsilon)}{t_{mix}^k (1-\epsilon)}=1.\end{align*}
\end{cor}  

  From Theorem \ref{theorem0} we can obtain the following weaker version, which will be used to show  the examples of section \ref{sec4}.
\begin{thm} \label{theorem0b} Assume  that $\Phi= \sum_{x \in S } \frac{\pi(x)}{\pi(S)}P_{x} (t_{S}>\tau_0)\downarrow 0$ as $k \rightarrow \infty$.  For $\epsilon \in (0,1)$, assume  $k$ large enough so that   $\pi(S)+\sqrt{\Phi}<1-\epsilon$.  Assume  conditions (C) and     that for large $k$ there  exist $\gamma$ and $\delta$ such that    (H) and 
$$\mathbb{E}_{ \partial D}(\tau_0)\prec\delta\ \sqrt{Var_{ \partial D}(\tau_0)} $$    and
that for $k$ large$$ \mathbb{E}_{ c}(\tau_0)\leq\delta\  \sqrt{Var_{ \partial D}(\tau_0)}\leq \frac{\mathbb{E}_{ c}(\tau_0) }{\sqrt{\Phi}}.$$
Then
\begin{align*} t_{mix}(\epsilon) \sim\frac{\mathbb{E}_{ c}(\tau_0) }{\sqrt{\Phi}}.\end{align*}
 \end{thm}
A few words about the proof of the two theorems and the structure of the paper. In order to find the asymptotic limits of the mixing time $t^k_{mix}$  we will calculate   upper and lower bounds for $t^k_{mix}(\epsilon)$. Under the conditions  of Theorem  \ref{theorem1}, for $0<\epsilon<1$ we will    show that for large $k$ 
\begin{align}\label{loup1}  \mathbb{E}_{ c}(\tau_0)-\gamma\sqrt{Var_{ c}(\tau_0)}\leq t^k_{mix}(\epsilon)\leq\mathbb{E}_{ c}(\tau_0)+\gamma\sqrt{Var_{ c}(\tau_0)} . \end{align}
Then the mixing time and the cutoff property follow because of (H)(a).

Under the conditions  of Theorem  \ref{theorem0}, we will show that for $k$ large enough so that  $\pi(S)+\sqrt{\Phi}\leq 1-\epsilon$
 \begin{align}\label{loup2}  \frac{\mathbb{E}_{ c}(\tau_0)-\gamma\sqrt{Var_{ c}(\tau_0)}}{\sqrt{\Phi}}\leq t^k_{mix}(\epsilon)\leq\frac{\mathbb{E}_{ c}(\tau_0)+\gamma\sqrt{Var_{ c}(\tau_0)}}{\sqrt{\Phi}}  \end{align}
when $k$ is large. Then the mixing time and the cutoff property follow again because of (H)(a).  

The lower bound of both (\ref{loup1}) and (\ref{loup2}) will be shown in   Lemma \ref{palLem} and Proposition \ref{newlower} of section 2 respectively. The upper bound  of (\ref{loup1}) and (\ref{loup2}) follows from Proposition \ref{prop3.2} for $A=1$ and $A=\sqrt{\Phi}<1$ respectively. 

\section{Lower bounds.} \label{lower}
In this section we present two  lower bounds for the mixing time $t^k_{mix}$.  The first lower bound is presented in   Lemma \ref{palLem}. Then under the condition  $\Phi=\sum_{x \in S } \frac{\pi(x)}{\pi(S)}P_{x} (t_{S}>\tau_0)<1$ for large $k$, we prove in Proposition \ref{newlower} a sharper lower bound.    The first lower bound for the mixing time follows.
\begin{lem}\label{palLem}Assume (C) and $\gamma\rightarrow \infty$ as $k\rightarrow \infty$. The following lower bound for the mixing time holds:
$$t_{mix}(\epsilon)\geq\mathbb{E}_{ c}(\tau_0)-\gamma\sqrt{Var_{ c}(\tau_0)} $$
for large $k$.
\end{lem}
The proof of   Lemma \ref{palLem}  is presented in \cite{P-S} (lower bound of Theorem 1.1 \cite{P-S}). We will use this bound in order to show a sharper lower bound on the following proposition.
 
Denote $\pi_S$ the restriction of $\pi$ to $S$, $\pi_S(A)=\pi(S\cap A)$ and $\mu_S (A)=\frac{\pi_S(A)}{\pi (S)}$.
In order to prove the second sharper lower bound we will use the approach  of \cite{L-P-W} for graphs  with one bottleneck.
 The main result  related  to the lower bound of the mixing time follows in the next proposition.
\begin{prop}\label{newlower} For $\epsilon \in (0,1)$, assume $k$ large enough so that $\sqrt{\Phi}+\pi(S)<1-\epsilon$. Then for $k$ large
\begin{align*} t^k_{mix}(\epsilon) \geq\frac{\mathbb{E}_{ c}(\tau_0)-\gamma\sqrt{Var_{ c}(\tau_0)}}{\sqrt{\Phi}}.\end{align*}
  \end{prop}
 \begin{proof}
  Since from Lemma \ref{palLem} we know that  for $k$ large
$t^k_{mix}\geq\mathbb{E}_{ c}(\tau_0)-\gamma\sqrt{Var_{ c}(\tau_0)} =t_S$, there exists an $m\geq 1$ such that $t^k_{mix(\epsilon )}=mt_S$.  In order to bound the total variation distance   $ \parallel \mu_SP^{mt_S}-\mu_S\parallel_{TV}$
we can use the following bound 
 \begin{align*} \parallel \mu_SP^{t}-\mu_S\parallel_{TV}  \leq  t \Phi(S), \ \ \ t \in \mathbb{N} \end{align*}
 where $\Phi(S)=\sum_{x\in S}\sum_{y \notin S}\mu_S (x)P(x,y),$ (see   (7.14)-(7.15) from  \cite{L-P-W}). We then obtain
    \begin{align}\label{eq1.12n} \parallel \mu_SP^{t^k_{mix}}-\mu_S\parallel_{TV}  \leq   t^k_{mix}\Phi(S)=m t_S \Phi(S) =m\Phi\ \end{align}
    where we have denoted $\Phi= t_S \Phi(S)$. We have
\begin{align*} \parallel \mu_S -\pi\parallel_{TV} =\max_{A\subset\Omega}\left\vert\mu_S (A)-\pi(A)\right\vert\geq \left\vert\mu_S (S^{c})-\pi(S^{c})\right\vert =\pi(S^{c})  \end{align*}
since $\mu_S (S^{c})=0$. So, using the triangular inequality,  we can write
\begin{align*} 1-\pi(S)\leq\parallel \mu_S -\pi\parallel_{TV} \leq\parallel \mu_SP^{t^k_{mix} }-\mu_S \parallel_{TV}+\parallel \mu_SP^{t^k_{mix} }-\pi\ \parallel_{TV}  \end{align*}
which leads to
$$1-\pi(S)\leq \parallel \mu_SP^{t_{mix} }-\mu_S \parallel_{TV}+\epsilon$$ If we bound the first term on the right hand 
by (\ref{eq1.12n}) we then have
\begin{align} 1-\pi(S)\leq m \Phi +\epsilon \label{eq1.12.1} \end{align}
which after  substituting $m=\frac{t^k_{mix(\epsilon )}}{t_{S}}$  gives the  following lower bound for $t^k_{mix(\epsilon )}$
\begin{align*} t^k_{mix(\epsilon )} \geq\frac{ t_S}{ \frac{\Phi}{1-\pi(S)-\epsilon}} .\end{align*}If we choose $k$ large enough so that \[\sqrt{\Phi}+\pi (S)+\epsilon \leq 1 \ \Rightarrow \ \frac{\Phi}{1-\pi(S)-\epsilon}\leq \sqrt{\Phi}\]
we then obtain the desired lower bound
\begin{align*} t^k_{mix(\epsilon )} \geq\frac{ t_S}{ \sqrt{\Phi}} .\end{align*}
\end{proof}
\section{Upper bound for the  cutoff case.} \label{upperC}

 In this section we  calculate the upper bound for the mixing time. Technics from \cite{P-S} will be   closely followed.
 We start with a technical lemma.
  \begin{lem}\label{neolem}Assume (C2) and  $\pi(S)+\sqrt{\Phi}<1-\epsilon$ for large $k$. Denote $t'=\frac{\mathbb{E}_{ c}(\tau_0)+\frac{\gamma}{2}\sqrt{Var_{ c}(\tau_0)}}{\sqrt{\Phi}}$, and define the set   $B=\{\tau_0<t'\}$. If   \begin{align*}(\star) \ \ \ \ \ t' \geq\mathbb{E}_{ \partial D}(\tau_0)+\delta \sqrt{Var_{ \partial D}(\tau_0)}\end{align*}   then
\begin{align}\label{setB}P_{x}( B^{c})\leq \frac{16}{\gamma^2}+\frac{1}{\delta^2}+\frac{4\mathbb{E}_{ \partial D}(\tau_0) }{\frac{\gamma\sqrt{Var_{ c}(\tau_0)}}{\sqrt{\Phi}}}.\end{align}
 \end{lem}\begin{proof}

 We will first consider  $ x\in S\cup B[0,z]$. To show (\ref{setB})  we will distinguish on the two different cases of graphs denoted as Case A and B, shown on figures 1 and 2 respectively.  
 
 For the Case A we have 
\begin{align}\label{inS}\nonumber P_{x}( B^{c})=&     P_x(\tau_0\geq t')\leq     P_{c}(\tau_0\geq t')=     P_{c}(\tau_0\geq \frac{\mathbb{E}_{ c}(\tau_0)+\frac{\gamma}{2}\sqrt{Var_{ c}(\tau_0)}}{\sqrt{\Phi}}) \\  \leq &P_{c}(\tau_0\geq \mathbb{E}_{ c}(\tau_0)+\frac{\gamma}{2}\sqrt{Var_{ c}(\tau_0)})\leq\frac{4}{\gamma^{2}}\end{align}
where above we used first that for $k$ large   $\sqrt{\Phi}\leq 1$ and then Chebyshev's inequality. 

We will show the same for every  $ x\in S\cup B[0,z]$ for graphs in  Case B. Define $\tau_1$ the time it gets to hit $[0,c]$ and $\tau_2$ the time it takes  to hit $0$ starting  from $X_{\tau_1}$. Then  
\begin{align}\label{inS1--}\nonumber P_{x}( B^{c})=&     P_x(\tau_0\geq t')=     P_{x}(\tau_0\geq \frac{\mathbb{E}_{ c}(\tau_0)+\frac{\gamma}{2}\sqrt{Var_{ c}(\tau_0)}}{\sqrt{\Phi}})\leq \\ \leq &P_{x}(\tau_1\geq\frac{\gamma}{4}\frac{ \sqrt{Var_{ c}(\tau_0)}}{\sqrt{\Phi}})+P_{x}(\tau_2\geq \frac{\mathbb{E}_{ c}(\tau_0)+\frac{\gamma}{4}\sqrt{Var_{ c}(\tau_0)}}{\sqrt{\Phi}})\end{align}
For the first term on the right hand side of (\ref{inS1--}) we can use Markov inequality to get 
\begin{align} \nonumber P_{x}(\tau_1\geq\frac{\gamma}{4}\frac{ \sqrt{Var_{ c}(\tau_0)}}{\sqrt{\Phi}}) \leq  \frac{4\mathbb{E}_{ x}(\tau_1)}{\frac{\gamma\sqrt{Var_{ c}(\tau_0)}}{\sqrt{\Phi}}}\end{align}
But for any  $x\in T_i$  $$\mathbb{E}_{ x}(\tau_1)\leq\mathbb{E}_{ \partial T_{i}}(n_{i})\leq\mathbb{E}_{ \partial D}(\tau_0) $$  because of (C3). This leads to 
\begin{align}\label{inS1--1} P_{x}(\tau_1\geq\frac{\gamma}{4}\frac{ \sqrt{Var_{ c}(\tau_0)}}{A}) \leq  \frac{4\mathbb{E}_{ \partial D}(\tau_0) }{\frac{\gamma\sqrt{Var_{ c}(\tau_0)}}{\sqrt{\Phi}}}\end{align}

For the second term on the right hand side of (\ref{inS1--}) we have
\begin{align}\label{inS1--2}\nonumber P_{x}(\tau_2\geq  \frac{\mathbb{E}_{ c}(\tau_0)+\frac{\gamma}{4}\sqrt{Var_{ c}(\tau_0)}}{\sqrt{\Phi}})\leq &P_{x}(\tau_2\geq \mathbb{E}_{ c}(\tau_0)+\frac{\gamma}{4}\sqrt{Var_{ c}(\tau_0)}) \\  \leq & P_{c}(\tau_0\geq\mathbb{E}_{ c}(\tau_0)+\frac{\gamma}{4}\sqrt{Var_{ c}(\tau_0)})  \leq  \frac{16}{\gamma^2}\end{align} 
  above we used  that $\sqrt{\Phi}<1$ for large $k$ and applied Chebyshev'´s inequality.  Finally, putting (\ref{inS1--1}) and (\ref{inS1--2}) in  (\ref{inS1--}) gives that for every $x\in S\cup B[0,z]$
\begin{align}\label{inS1} P_{x}( B^{c}) \leq &\frac{16}{\gamma^{2}}+\frac{4\mathbb{E}_{ \partial D}(\tau_0) }{\frac{\gamma\sqrt{Var_{ c}(\tau_0)}}{\sqrt{\Phi}}}\end{align}
 for graphs as in Case B.
From (\ref{inS}) and (\ref{inS1}) we obtain that in both Cases A and B, for every  $x\in S\cup B[0,z]$ one has
\begin{align}\label{inS1++} P_{x}( B^{c}) \leq &\frac{16}{\gamma^{2}}+\frac{4\mathbb{E}_{ \partial D}(\tau_0) }{\frac{\gamma\sqrt{Var_{ c}(\tau_0)}}{\sqrt{\Phi}}}\end{align}
For both Cases A and B, when  $x\in T_{0}$ we have 
\begin{align}\nonumber  P_{x}( B^{c})=     P_x(\tau_0\geq t')\leq       P_{\partial D}( \tau_0\geq t') \end{align}
For $k$ large so that ($\star$) holds, we can bound 
\begin{align}  \label{inT_0}P_{x}( B^{c}) \leq   P_{\partial D}(\tau_0\geq\mathbb{E}_{ \partial D}(\tau_0)+\delta\ \sqrt{Var_{ \partial D}(\tau_0)})\leq\frac{1}{\delta^{2}}\end{align}
where  above we  used  Chebyshev's  inequality. 
From (\ref{inS1++}) and (\ref{inT_0}) we obtain    (\ref{setB}) for every $x\in D\cup S$.
 \end{proof}
The main result about the upper bound of the mixing time follows.  
\begin{prop}\label{prop3.2} Assume  (C) and  that for large $k$   (H) and    \begin{align*}(\star) \ \ \ \ \ \frac{\mathbb{E}_{ c}(\tau_0)+\frac{\gamma}{2}\sqrt{Var_{ c}(\tau_0)}}{A}\geq\mathbb{E}_{ \partial D}(\tau_0)+\delta \sqrt{Var_{ \partial D}(\tau_0)}\end{align*}for some  $A\leq 1$. Then 
$$t^k_{mix}\leq\frac{\mathbb{E}_{ c}(\tau_0)+\gamma\sqrt{Var_{ c}(\tau_0)}}{A}$$
 \end{prop}
\begin{proof}Denote $t=\frac{\mathbb{E}_{ c}(\tau_0)+\gamma\sqrt{Var_{ c}(\tau_0)}}{A}$. We will consider the following coupling. Assume $X_0=x $ and $Y_0\sim \pi$.
We let $X$ and $Y$ move independently until the first time $X$ hits $0$. Then they still continue both moving independently until the moment they collide or reach the same level at $T_0$. In this case the coupling changes to the following. $X$ keeps moving as an aperiodic random walk while $Y$ moves closer or further from $0$ if $X$ moves closer or further respectively. Define  $\tau$ to be the coupling time.

 Then  define     $E$ to be  the event that after hitting $0$ for the first time it reaches  $\partial D$ of $T_ 0$ before hitting $z$, i.e.  $$E=\{\tau^*_{\partial D}<\tau^*_{z}\} $$as well as the the events 
$$A_{L_z}= \{L_z>z\} . $$
If we define  $$S=inf\{s\geq \tau^*_{\partial D}:x_s=0\}$$
then on the  event   $E$, the quantity $S-\tau_0$ is   dominated by  
\begin{align*} \sum_{i=1}^{L_z}\theta_{i}+\sum_{i=1}^{ L_z}\sum_{j=1}^{G_i} \lambda_{i,j}+\zeta+\xi\end{align*} (see \cite{P-S}) where we recall 
 $(\theta_{i})_i$ are iid variables distributed as the length of a random walk on $B[0,z]$ starting from $0$ conditioning not to hit $z$ and $(\lambda_{i,j})$ are iid random variables distributed as the length of a random walk from $0$ on $T_0$ conditioning not to hit $\partial T_0$. $(G_i)_i$ is a random variable with  probability of success $\frac{\{\# j :j \sim 0, j\in B[0,z]\}}{\{\# j :j \sim 0\}}$  and $\xi$ is a random variable distributed as the commute  time between  $\partial T_0$ and $0$. If we use Wald's identity we obtain
\begin{align} \label{3.5} \mathbb{E}[(S-\tau_0)\mathcal{I}_E] \leq & \mathbb{E}[ L_z]\mathbb{E}[\theta_i]+\mathbb{E}[ L_z]  \mathbb{E}[G_i] \mathbb{E}[\lambda_{i,j}] +\mathbb{E}[\zeta]  +\mathbb{E}[\xi]  \lesssim \\ \lesssim \nonumber & \mathbb{E}_{ \partial D}(\tau_0)+\mathbb{E}[ L_z]  \mathbb{E}[\lambda_{i,j}]+ \mathbb{E}[ L_z]\mathbb{E}[\theta_i]    \end{align}
where above   we used (C3) for  $\mathbb{E}[\xi] = \mathbb{E}_{ \partial D}(\tau_0)$ and $\mathbb{E}[\zeta]=\mathbb{E}_0(\tau_{ \partial D})$.

We compute
 \begin{align*} P(\tau>t)  \leq  P(\tau>t, A_{ L})+\frac{1}{h(z)} \leq P(\tau>t, A_{ L_z}, E)+\frac{1}{h(z)}+p^{z} \end{align*}
where in the first inequality  we used  (C2) and in the second that $P(E^c \mid L_z)\leq p^{L_z}$. Then we obtain
 \begin{align*} P(\tau>t) & \leq   P(\tau>t, A_{ L_z}, E)+\frac{1}{h(z)}+p^{z} \end{align*}
 If we use Lemma \ref{neolem} the last one can be bounded by
 \begin{align}\label{lastref1} P(\tau>t) & \leq   P(\tau>t, A_{ L_z}, E,B)+\frac{16}{\gamma^{2}}+\frac{1}{\delta^{2}}+\frac{4\mathbb{E}_{ \partial D}(\tau_0) }{\frac{\gamma\sqrt{Var_{ c}(\tau_0)}}{A}}+\frac{1}{h(z)}+p^{z} \end{align}
Now define $F=\{S-\tau_0>\frac{\frac{\gamma}{2}\sqrt{Var_{ c}(\tau_0)}}{A}\}$ and  $M=\{Y_{\tau_0} \in D\}$.  Then  $P(M^c)=o(1)$ as   $k\rightarrow \infty$ since at time  $\tau_0$ the random walk   $Y$ is stationary, and so, because of (C1) we have that   the  stationary probability of  $T_{0}$ is $1-o(1)$. We observe that on the events $E$ and $M$ the two walks $X$ and $Y$ must have  coalesced by time $S$. Therefore
$$B\cap F^c\subset \{\tau<t\}$$
  This implies that
 \begin{align*}   P(\tau>t, A_{ L_z}, E,B)&\leq P(\tau>t, A_{ L_z}, E,M,B)+o(1)=P(\tau>t,A_{ L_z}, E,M,B,F)+o(1)
\leq \\ &\leq P(E,F)+o(1)=   P(E,\{S-\tau_0>\frac{\frac{\gamma}{2}\sqrt{Var_{ c}(\tau_0)}}{A}\}) +o(1) \end{align*}
 From Markov  inequality
 \begin{align*}   P(\tau>t, A_{ L_z}, E, B)\leq \frac{ 2\mathbb{E}[(S-\tau_0)\mathcal{I}_{ E}]}{\frac{\gamma\sqrt{Var_{ c}(\tau_0)}}{A}} +o(1)  \end{align*}
 If we now use  (\ref{3.5})  we obtain
 \begin{align}\label{lastref2}   P(\tau>t, A_{ L_z}, E,B)\leq2 \frac{ \mathbb{E}_{ \partial D}(\tau_0)+\mathbb{E}[ L_z]  \mathbb{E}[\lambda_{i,j}]+ \mathbb{E}[ L_z]\mathbb{E}[\theta_i]  }{\frac{\gamma\sqrt{Var_{ c}(\tau_0)}}{A}}  +o(1) \end{align}
 Combining together (\ref{lastref1}) and (\ref{lastref2})
we get \begin{align*}   P(\tau>t)\leq &2\frac{ \mathbb{E}_{ \partial D}(\tau_0)+\mathbb{E}[ L_z]  \mathbb{E}[\lambda_{i,j}]+ \mathbb{E}[ L_z]\mathbb{E}[\theta_i]  }{\frac{\gamma\sqrt{Var_{ c}(\tau_0)}}{A}} +\frac{16}{\gamma^{2}}+\frac{1}{\delta^{2}}+\frac{4\mathbb{E}_{ \partial D}(\tau_0) }{\frac{\gamma\sqrt{Var_{ c}(\tau_0)}}{A}}+\\ & +\frac{1}{h(z)}+p^{z}+o(1)  \end{align*}
If one takes  under  account that \begin{align*}    \max_x \vert \vert P^t(x,\dot ) - \pi \vert \vert \leq P(\tau>t)  \end{align*}
(see \cite{L-P-W}) we eventually obtain 
 \begin{align*} \max_x \vert \vert P^t(x,\dot ) - \pi \vert \vert\leq &2\frac{ \mathbb{E}_{ \partial D}(\tau_0)+\mathbb{E}[ L_z]  \mathbb{E}[\lambda_{i,j}]+ \mathbb{E}[ L_z]\mathbb{E}[\theta_i]  }{\frac{\gamma\sqrt{Var_{ c}(\tau_0)}}{A}} +\frac{16}{\gamma^{2}}+\frac{1}{\delta^{2}}+\frac{4\mathbb{E}_{ \partial D}(\tau_0) }{\frac{\gamma\sqrt{Var_{ c}(\tau_0)}}{A}}+\\ & +\frac{1}{h(z)}+p^{z} +o(1)   \end{align*}
 Because of (H)(b), the fact that $h(z),z\rightarrow \infty$ as $k\rightarrow \infty$ (recall $p<1$)    and that $A<1$ for large $k$, we obtain   $t_{mix}\leq\frac{\mathbb{E}_{ c}(\tau_0)+\gamma\sqrt{Var_{ c}(\tau_0)}}{A}$.
 \end{proof}
 \section{Paradigms}\label{sec4}
 In this section we present examples for the two main theorems. At first in paradigm 1 an example that satisfies the conditions of Theorem \ref{theorem1} is presented. Then in sections \ref{secpar2} and \ref{secpar3}, two examples that satisfy the conditions of Theorem \ref{theorem0} are presented.  From these last two, at the first one,  paradigm 2, we consider a graph with 
$$\mathbb{E}_{ \partial D}(\tau_0)\leq\sqrt{Var_{ \partial D}(\tau_0)}$$  while at the second one, paradigm 3, a graph with 
$$\mathbb{E}_{ \partial D}(\tau_0) > \sqrt{Var_{ \partial D}(\tau_0)}.$$ 
For both examples we establish $\sqrt{\Phi}^{-1}\rightarrow\infty$ as $k\rightarrow \infty$.
\subsection{Paradigm 1}\label{secpar1} This is the example presented in \cite{P-S}. 

We will show that condition  (H1) is satisfied. Since in the graph presented one has 
$$ \sqrt{Var_{ \partial D}(\tau_0)}\lesssim\ \mathbb{E}_{ \partial D}(\tau_0)\text{ \ and \ }  \sqrt{Var_{ c}(\tau_0)}\lesssim\ \mathbb{E}_{ c}(\tau_0)  $$ (see \cite{P-S})  it is sufficient to show that  
\begin{align}\label{examp1.1} \mathbb{E}_{ \partial D}(\tau_0)\lesssim\sqrt{Var_{ c}(\tau_0)} .\end{align}
Since $N$ is the size  of binary tree $T_0$, one has that $N=\sum_{j=0}^{l}2^j $, were $l$ is the number of levels of $T_0$, i.e. $l=\log(N-1)-1$.  Concerning (\ref{examp1.1}), one can think of the random walk from the leafs $ \partial T_{0}$ of the binary tree to the origin $0$ as a walk from $0$ to $l$ on $[0,l]\cap \mathbb{Z}$, with a reflective boundary at $0$ and probabilities $\frac{2}{3}$ and $\frac{1}{3}$ towards the left and the right respectively at any other point of the line. Since, for this  one dimensional random walk   the hitting time satisfies
$$l^2\lesssim\mathbb{E}_{ 0}(\tau_l)\lesssim l^{3}$$
we obtain for $T_0$ that
\begin{align}\label{examp1.2}l^2\asymp(\log(N-1)-1)^{2}\lesssim\mathbb{E}_{ \partial D}(\tau_0)\lesssim(\log(N-1)-1)^{3}\asymp l^3.\end{align}
On the other  hand we know that (see \cite{P-S}) 
\begin{align}\label{examp1.3}\sqrt{Var_{ c}(\tau_0)}\asymp N\sqrt{k}.\end{align}
From (\ref{examp1.2}) and (\ref{examp1.3}) inequality (\ref{examp1.1}) follows for appropriately large $N=n_k^3$ so that $$(\log(N-1)-1)^{3}\lesssim N\sqrt{k}$$
or if we substitute  $N=n_k^3$ 
$$n_k^3< e^{n_k k^\frac{1}{6}+1}$$
which is true for $n_k$ sufficiently large. The rest of the conditions are easily verified directly from \cite{P-S}.

\subsection{Paradigm 2}\label{secpar2}
We will construct a graph based  partly on the graph $\mathcal{T}$ presented in \cite{P-S}. Let $n_k=2^{2^k}$ and consider the line $[0,n_k+1]$. Then for all $j \in \{[\frac{k}{2}],...,k\}$ place   a binary tree at distance $n_j$ from the origin consisting of $\frac{N}{ n_j}$ vertices. We denote this construction as $  G_0$. In this way, the part of $G_0$ contained between  $[0,n_k]$, i.e. $G_{0}\smallsetminus\{n_{k}+1\}$, is equal with  $\mathcal{T}\smallsetminus T_{0}$ from \cite{P-S}, where $T_0$ is the big binary  tree at $0$ of $\mathcal{T}$. Then consider $r$ identical copies of $G_0$ and glue them together at $0$ and $n_k+1$ as shown on figure \ref{parad2}. We also consider $q$ copies of a line $[-m,0]$ and connect them with the previous construction at $0$ and together at $-m$. In this way, we can consider $T_0$ to be the part of graph between $[-m,0]$ with $\partial T_{0}=-m$, the bottleneck $B[0,z]$ to be the part of the graph between $0$ and $n_{[\frac{k}{2}]}$, while $S$ is the part between $n_{[\frac{k}{2}]}$ and $n_k+1$. In this way  $c=n_k+1$ and $z=n_{[\frac{n_k}{2}]}$.
 \begin{figure} 
\begin{center}\includegraphics[width=12.25cm,height=5.25cm]{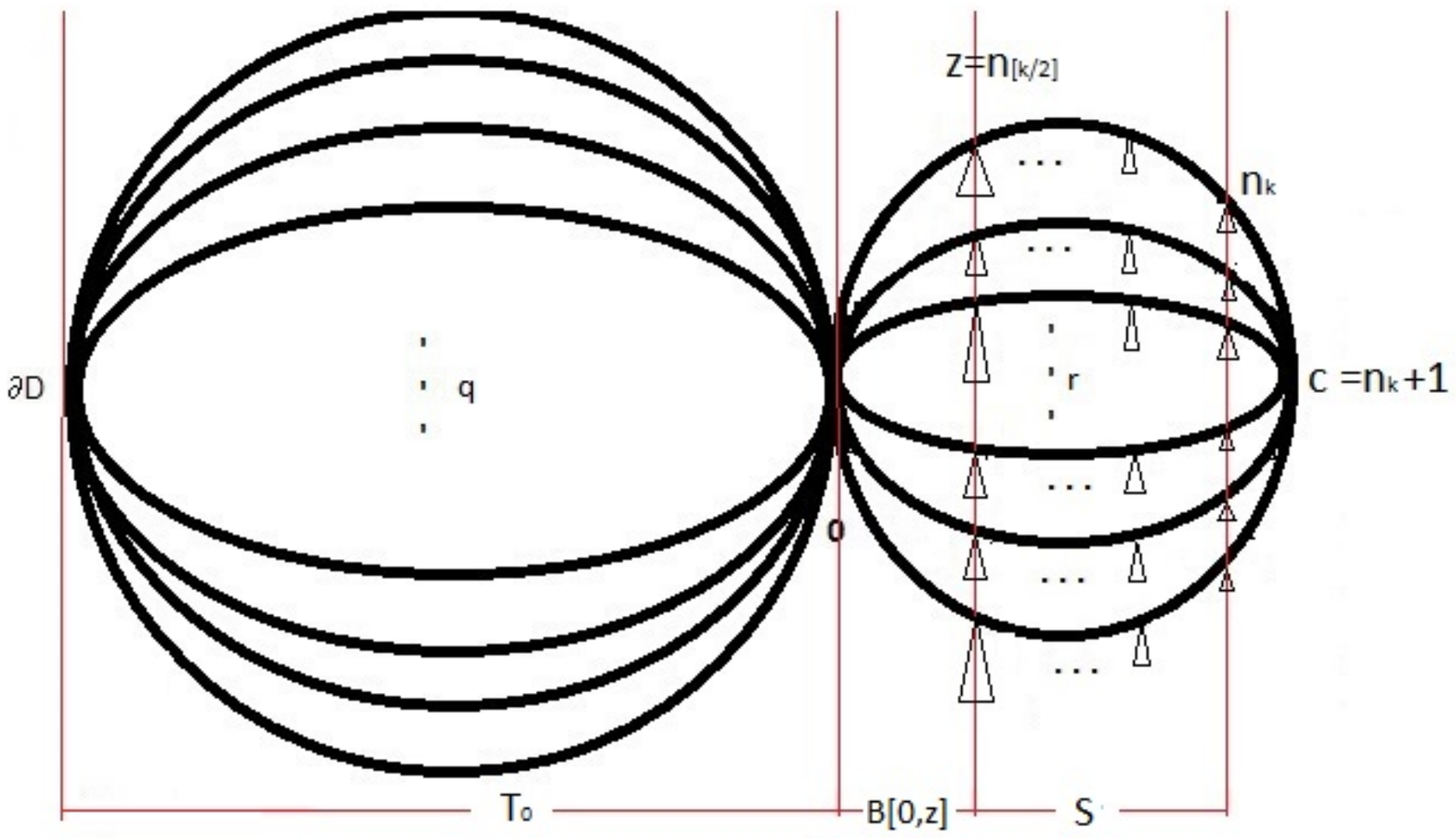}
\caption[paradigm2]{Paradigm 2}
\end{center}\label{parad2}
\end{figure}
 
We will determine $m,r,q,\gamma,\delta$ so that the conditions of Theorem \ref{theorem0b} hold.
But first, we will place conditions on the graph so that condition (H1) of Theorem \ref{theorem1} is not true. If  $X$ is a simple random walk on the interval $[-m,0]$ staring from $-m$,  then the mean and the variance of the time it gets to reach  $0$ from $m$ for the first time are $\mathbb{E}_{ -m}(\tau_0)=m^2$ and $Var_{ -m}(\tau_0)=m^4$ respectively.   Since  $D$ consists of $q$ lines $[-m,0]$ connected on $0$ and $-m$, and when at $-m$ we move at any of the $q$ branches $[-m,0]$ with the same  probability $\frac{1}{q}$ we obtain that 
$$\mathbb{E}_{ \partial D}(\tau_0)=m^{2} \text{\ \ \ \ and \ \ \ \ } Var_{ \partial D}(\tau_0)=m^4.$$
From \cite{P-S} we know that  
\begin{align*}\sqrt{Var_{ c}(\tau_0)}\asymp N\sqrt{k} \text{ \ \ \ \ \  and \ \   \ \ \ }\mathbb{E}_{ c}(\tau_0) =6Nk+o(N\sqrt{k})\asymp6Nk
.\end{align*}
If one chooses
\begin{align}\label{examp2not}m^{2} =6Nk
\end{align}
 then (H1) is not true since $\delta\sqrt{Var_{ \partial D}(\tau_0)}>\mathbb{E}_{ c}(\tau_0)$ for any $\delta= k^ t$ with $t>0$, and $k$ large. 
 We also notice that for   any $\delta= k^ t$ with $t>0$  $$\mathbb{E}_{ \partial D}(\tau_0)\prec\delta\sqrt{Var_{ \partial D}(\tau_0)}.$$
  Furthermore,  for any  $\gamma= k^ p$ with $p<\frac{1}{2}$ we get
\begin{align}\gamma\sqrt{Var_{ c}(\tau_0)}\prec\mathbb{E}_{ c}(\tau_0).\label{examp2.1}\end{align}
In addition,  from (\ref{examp2.1}) we get that
\begin{align}\label{examp2.2} t_S=\mathbb{E}_{ c}(\tau_0)-\gamma\sqrt{Var_{ c}(\tau_0)}\prec6Nk
.\end{align}
At first we  will show  that for appropriate $r>0$ one can obtain   $\sqrt{\Phi}<\frac{1}{k^{s}}$ for some $s>0$. 

Thus,  it suffices to show that $\Phi=t_{S}\sum_{x\in S}\sum_{y \notin S}\mu_S (x)P(x,y)<\frac{1}{k^{2s}}$.
For that, we  compute 
\begin{align}\label{examp2.3} \sum_{x\in S}\sum_{y \notin S}\mu_S (x)P(x,y)=\frac{q}{(q+r)\left\vert  S \right\vert}
\end{align}
where the   size  of $S$ is 
\begin{align}\label{examp2.4} \left\vert  S\right\vert=r(n_{k}+1+N(\sum_{j\in\{[k/2],...,k\}}\frac{1}{n_{j}}))\asymp rN( \sum_{j\in\{[k/2],...,k\}}\frac{1}{n_{j}}).\end{align}
Combining together (\ref{examp2.2}), (\ref{examp2.3}) and (\ref{examp2.4}) we get
 $$\Phi\asymp\frac{6Nkq( \sum_{j\in\{[k/2],...,k\}}\frac{1}{n_{j}})^{-1}}{(q+r)rN} .$$
So, if we   choose 
\begin{align}\label{examp2.5}r=12k^{2s+1} ( \sum_{j\in\{[k/2],...,k\}}\frac{1}{n_{j}})^{-1}\end{align} we obtain $\sqrt{\Phi}\downarrow 0$
as  $k\rightarrow \infty$. We will now determine parameters  so that conditions (C) are satisfied.  
For  (C1) we compute 
$$ \pi (B[0,z]\cup S)=\frac{\left\vert  S\right\vert+\left\vert  B[0,z]\right\vert  }{\left\vert  S\right\vert+\left\vert  B[0,z]\right\vert+\left\vert  T_{0}\right\vert} .$$ For this to vanish as $  k\rightarrow \infty $ we need
 \begin{align}\label{examp2.6-1}\left\vert  S\right\vert+\left\vert  B[0,z]\right\vert\prec\left\vert   T_{0}\right\vert.\end{align}
 But $\left\vert  S\right\vert+\left\vert  B[0,z]\right\vert\asymp rN( \sum_{j\in\{[k/2],...,k\}}\frac{1}{n_{j}})$ and $\left\vert  T_{0}\right\vert\asymp q(m-2)$. 
For  (C1) to hold true we need $q$ large  enough so that 
\begin{align*}\frac{rN}{m-2}( \sum_{j\in\{[k/2],...,k\}}\frac{1}{n_{j}})\prec\ q. \end{align*}
   If we replace $r$ by (\ref{examp2.5}) and  $m$ by (\ref{examp2not}) we derive
    
\begin{align}\label{examp2.6}\frac{12k^{2s+1}N}{\sqrt{6Nk}-2}\prec\ q. \end{align}
   For the condition (C2), we recall that $L_{n_{[k/2]}}=\sum_{s=\tau_0}^{\tau^*_{n_{[k/2]}}}\mathcal{I}(x_{s-1}\notin T_0,x_s \in T_0)$. Then, since we move to every branch of the bottleneck with the same probability and the $r$ branches $[0,n_{[k/2]}]$ are identical,  $L$ has the geometric distribution with parameter $\frac{1}{n_{[k/2]}}$.  This leads to
\begin{align*} 
(C2) \  \  \   \  \  \  \  \  \ P(L_{n_{[k/2]}}\leq n_{[k/2]})\leq\frac{1}{  n_{[k/2]}}
.\end{align*} 
Concerning (C3), the first assertion $\mathbb{E}_0(\tau_{ \partial D})\leq \mathbb{E}_{ \partial D}(\tau_0)$ is trivially true with an equality since $T_0$ is by construction symmetric.   It remains to determine conditions for  (H)(b)  and the modification of condition (H2) of Theorem \ref{theorem0b}. For (H)(b) we first notice that     \begin{align}\label{reduce}   \mathbb{E}[\theta_i]&\lesssim\ \mathbb{E}[\lambda_{i,j}] . \end{align}
  Since  $(\theta_{i})_i$ are distributed as the length of a random walk on $r$ identical lines, with equal probability, of length  $z$ conditioning   not to hit $z$ and $(\lambda_{i,j}{i,j})$ are iid random variables distributed as the  length of a random walk on $q$ identical lines, with equal probability, of length $m$ conditioning not to hit $m$,  the last inequality is true if and only if 
\begin{align}\label{examp2.7}m\geq z=n_{[\frac{k}{2}]}\end{align}
which is true from (\ref{examp2not}). Because of (\ref{reduce}), (H)(b) is reduced to  
\begin{align}\label{examp2.8}
      \mathbb{E}_{ \partial D}(\tau_0)+\mathbb{E}[ L]  \mathbb{E}[\lambda_{i,j}]\prec\frac{\gamma\sqrt{Var_{ c}(\tau_0)}}{\sqrt{\Phi}}.\end{align}
We can bound the left hand side by $$\mathbb{E}_{ \partial D}(\tau_0)+\mathbb{E}[ L]  \mathbb{E}[\lambda_{i,j}]\asymp m^{2}+mn_{[\frac{k}{2}]}\lesssim m^{2}=6Nk$$  where above  we made use of   (\ref{examp2.7}) and (\ref{examp2not}). 
So, for   (\ref{examp2.8}) to be true we need
\begin{align*}6Nk\prec\frac{\gamma N\sqrt{k}}{\sqrt{\Phi}}.\end{align*}
Since $\sqrt{\Phi}\leq \frac{1}{k^s}$ the last one is satisfied for every $s$ such that  $\frac{1}{2}-p< s$
which is true for every $s\geq 1$. Finally for (H2) we need for large $k$\begin{align*} \mathbb{E}_{ c}(\tau_0)\leq\delta\  \sqrt{Var_{ \partial D}(\tau_0)}\leq \frac{\mathbb{E}_{ c}(\tau_0) }{\sqrt{\Phi}}\end{align*}
which is trivially  true for any  $t\leq s$.
  
\subsection{Paradigm 3} \label{secpar3} 
\begin{figure}
\begin{center} \includegraphics[width=12.25cm,height=5.25cm]{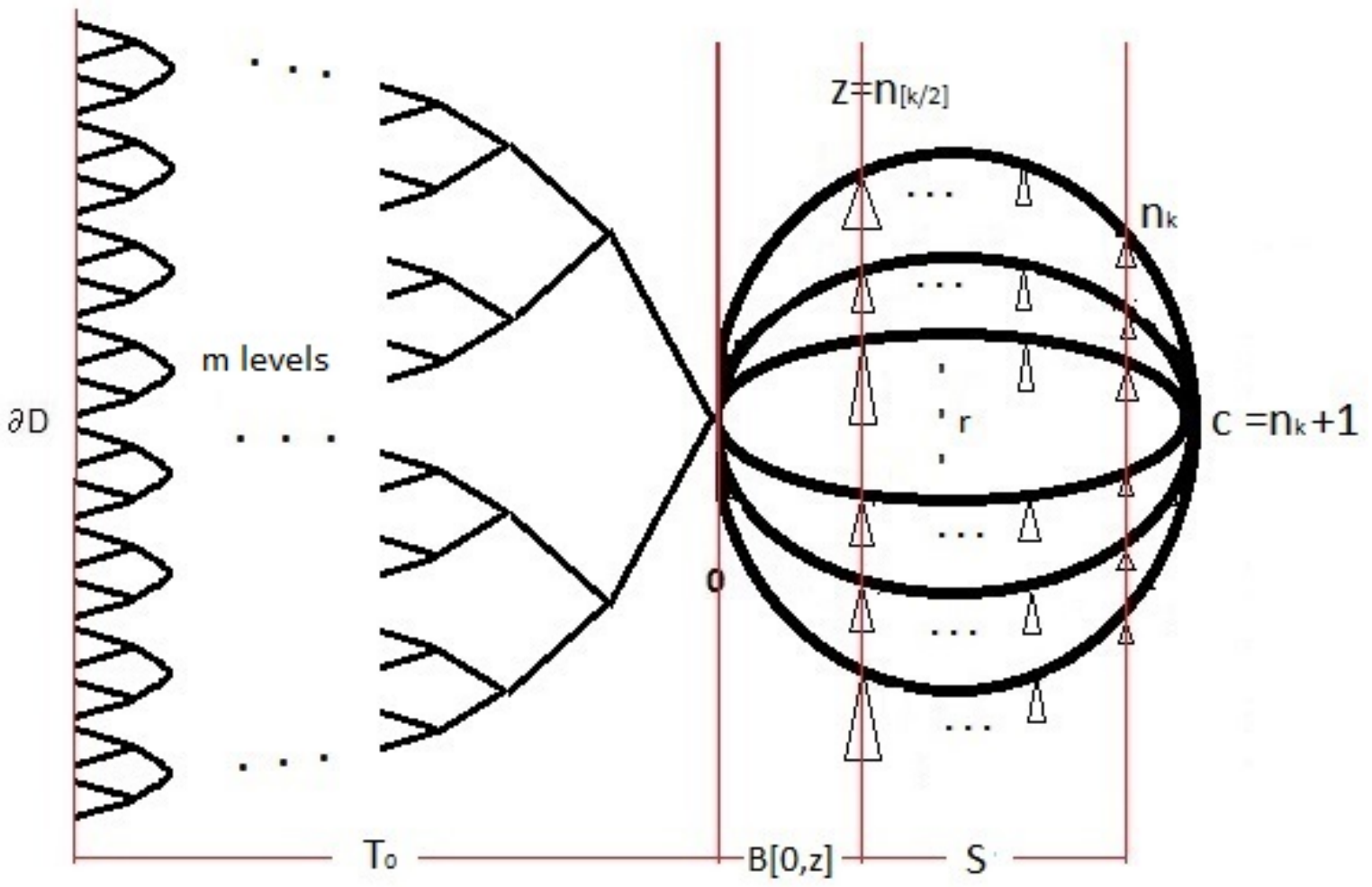}
\caption[paradigm3]{Paradigm 3}
\end{center}\label{parad3}
\end{figure}
We will  construct  a graph based  again on the example presented in \cite{P-S}. For $T_0$ we will consider a binary tree of size $M$. The remaining  part of the graph $B[0,z]\cup S$    will be the same with that of paradigm 2, as shown on  figure \ref{parad3}. 
 
We will  determine $r,\gamma,\delta,m, M$ so that the conditions of Theorem \ref{theorem0b} hold.
As in the previous example we will start by placing conditions so that (H1) of Theorem \ref{theorem1} is not true. Since, in a binary tree  
\begin{align}\label{bintree} \sqrt{Var_{ \partial D}(\tau_0)}\prec \mathbb{E}_{ \partial D}(\tau_0)\end{align}
it is sufficient to have   
 \begin{align}\label{examp3not2} \mathbb{E}_{ \partial D}(\tau_0)> \mathbb{E}_{ c}(\tau_0)=6Nk. \end{align}
Since from (\ref{examp1.2})
we know that   \begin{align*}  m^{2} \lesssim\mathbb{E}_{ \partial D}(\tau_0)  \end{align*}
we can choose 
\begin{align}\label{examp3not}  m >\sqrt{Nk}. \end{align}
As in paradigm 2, since the part $B[0,z]\cup S$ is common in the two examples, condition (C2) and       
$$\gamma\sqrt{Var_{ c}(\tau_0)}\prec\mathbb{E}_{ c}(\tau_0)$$ 
   for $\gamma= k^ p$ with $p<\frac{1}{2}$ are satisfied. Similarly, 
   \begin{align}\label{examp3.1} t_S\prec6Nk
.\end{align}
We will determine $s>0$ such that  $\sqrt{\Phi}\leqslant n_k^{-s}$.  Similarly to the previous example,  it suffices to show that $\Phi=t_{S}\sum_{x\in S}\sum_{y \notin S}\mu_S (x)P(x,y)<n_k^{-2s} $.
 To determine parameters for $\Phi<n_k^{-2s} $  we compute 
\begin{align}\label{examp3.2} \sum_{x\in S}\sum_{y \notin S}\mu_S (x)P(x,y)=\frac{2}{(2+r)\left\vert  S \right\vert}
\end{align}
where the   size of $S$ is 
\begin{align}\label{examp3.3} \left\vert  S\right\vert\asymp rN( \sum_{j\in\{[k/2],...,k\}}\frac{1}{n_{j}}).\end{align}
Combining together (\ref{examp3.1}), (\ref{examp3.2}) and (\ref{examp3.3}) we get
 $$\Phi\asymp\frac{6Nk2( \sum_{j\in\{[k/2],...,k\}}\frac{1}{n_{j}})^{-1}}{(2+r)rN}\leq \frac{6 k( \sum_{j\in\{[k/2],...,k\}}\frac{1}{n_{j}})^{-1}}{r}.$$
 We can then choose  \begin{align}\label{examp3.4}r=6kn_{k}^{2s}( \sum_{j\in\{[k/2],...,k\}}\frac{1}{n_{j}})^{-1}.\end{align} Furthermore, we can request for $s$ to be sufficiently large so that 
$$n_{k}^{s}>\frac{m^{3}}{6kN}>\frac{\sqrt{Nk}}{6Nk}=\frac{1}{6}k^{1/2}n_{k}^{1/2}$$
for large $k$. From the last one and the right hand side of  (\ref{examp1.2}) $\mathbb{E}_{ \partial D}(\tau_0)<m^3$, we obtain      
 \begin{align}\label{examp3not3}\frac{\mathbb{E}_{ c}(\tau_0) }{\sqrt{\Phi}}>\mathbb{E}_{ \partial D}(\tau_0) . \end{align}
Concerning  (C1) we require (\ref{examp2.6-1}): $\left\vert  S\right\vert+\left\vert  B[0,z]\right\vert\prec\left\vert   T_{0}\right\vert$.
 We have  
 $$\left\vert  S\right\vert+\left\vert  B[0,z]\right\vert\asymp rN( \sum_{j\in\{[k/2],...,k\}}\frac{1}{n_{j}})  =6n_{k} ^{2s}kN $$
 where above we used (\ref{examp3.4}).   Since  the size of he binary tree  $T_0$ is  $\vert T_0\vert =M$ for  (\ref{examp2.6-1}) we need $M\asymp e^{m}$, i.e. $m$, large enough
so that    \begin{align*} n_{k} ^{2s}kN\prec M
.\end{align*}
 Concerning (C3), both assertions are  trivially true. For inequality (H)(b) we can compute  
 \begin{align*}      \mathbb{E}_{ \partial D}(\tau_0)+\mathbb{E}[ L]  \mathbb{E}[\lambda_{i,j}]+\mathbb{E}[ L]  \mathbb{E}[\theta_{i}]\prec m^{3}  \end{align*}   and $$\frac{\gamma\sqrt{Var_{ c}(\tau_0)}}{\sqrt{\Phi}}=n_{k} ^{s}k^{p+\frac{1}{2}}N.$$ From the last two we obtain (H)(b) for $s$ large enough such that $m^{3}<n_{k} ^{s}k^{p+\frac{1}{2}}N$. 
 It remains to show the following two inequalities. 
$$\mathbb{E}_{ \partial D}(\tau_0)\prec\delta\ \sqrt{Var_{ \partial D}(\tau_0)} $$  
and  
$$ \mathbb{E}_{ c}(\tau_0)\leq\delta\  \sqrt{Var_{ \partial D}(\tau_0)}\leq \frac{\mathbb{E}_{ c}(\tau_0) }{\sqrt{\Phi}}.$$
However, since    (\ref{bintree}), (\ref{examp3not2}) and (\ref{examp3not3})  hold, for  the last two inequalities to be true we need to choose $\delta$ in the following range
$$\frac{\mathbb{E}_{ \partial D}(\tau_0)}{\ \sqrt{Var_{ \partial D}(\tau_0)}}<\delta<\frac{\mathbb{E}_{ c}(\tau_0) }{\sqrt{\Phi}\  \sqrt{Var_{ \partial D}(\tau_0)}}.$$

\end{document}